\documentclass[a4paper,10pt]{article}
\usepackage[utf8]{inputenc}

\usepackage{amsmath,amsthm,amssymb}
\usepackage[space,noadjust,nocompress]{cite} 

\newcommand{\xdeta}{\zeta}
\newcommand{\NA}{N(A)}
\newcommand{\RA}{R(A)}

\DeclareMathOperator*{\wlim}{{}_{\rm w} lim\, }

\newcommand{\N}{\mathbb{N}}
\newcommand{\R}{\mathbb{R}}

\newcommand{\xn}{x_n}
\newcommand{\An}{A_n}
\renewcommand{\And}{{A_n^\dagger}}

\newcommand{\xmn}{x_{n,m}}
\newcommand{\xmnk}{x_{n_k,m_k}}
\newcommand{\Amn}{A_{n,m}}
\newcommand{\Amnd}{{A_{n,m}^\dagger}}
\newcommand{\Amns}{{A_{n,m}^*}}
\newcommand{\Amnsd}{{(A_{n,m}^*)}^\dagger}
\newcommand{\Amndsind}[1]{{(A_{#1}^*)}^\dagger}

\newcommand{\Upmn}{A^*\Amn}
\newcommand{\Upmns}{\Amns A}

\newcommand{\mim}{{\infty,m}}
\newcommand{\xm}{x_{\mim}}
\newcommand{\Am}{A_{\mim}}
\newcommand{\Amd}{{A_{\mim}^\dagger}}

\newcommand{\Ands}{{(A_n^*)}^\dagger}
\newcommand{\Andsind}[1]{{(A_{#1}^*)}^\dagger}

\newcommand{\Ans}{A_n^*}
\renewcommand{\And}{{A_n^\dagger}}

\newtheorem{theorem}{Theorem}
\newtheorem{lemma}[theorem]{Lemma}
\newtheorem{remark}{Remark}

\newtheorem{proposition}[theorem]{Proposition}
\newtheorem{corollary}[theorem]{Corollary}
\newtheorem{example}{Example}

\newcommand{\xd}{{x^\dagger}}

\title{Projection methods for ill-posed problems revisited}
\author{Stefan Kindermann\footnote{Industrial Mathematics Institute, Johannes Kepler University Linz, 
Alternbergergstra{\ss}e~69, 4040 Linz (kindermann@indmath.uni-linz.ac.at).} }
\date{}

\begin{document}
\maketitle
\begin{abstract}
The discretization of   least-squares problems for linear ill-posed operator equations 
in Hilbert spaces is considered.
The main subject of this article concerns conditions for convergence of the associated 
discretized minimum-norm least-squares solution to the exact solution 
using exact attainable data. The two cases of global convergence
(convergence for all exact solution) or local convergence (convergence for a specific exact solution)
are investigated.
We review the existing results  and prove new equivalent condition
when the discretized solution always  converges to the exact solution. 
An important tool is to recognize the discrete solution operator as 
oblique projection. Hence, global convergence can be characterized
by  certain subspaces having uniformly bounded angles. We furthermore derive practically 
useful conditions when this holds and put them into the context of known results. 
For local convergence we generalize results on the characterization of weak or strong
convergence and state some new sufficient conditions.
We furthermore provide an example of a bounded sequence of discretized solutions 
which does not converge at all, not even weakly.
\end{abstract}

\section{Introduction}
We study the role of discretization in the use of solving 
ill-posed linear operator equations in Hilbert spaces. 
Consider an ill-posed problem in Hilbert spaces
\begin{equation}\label{main}  A x = y, \end{equation}
where $A: X \to Y$ is continuous and  equation \eqref{main} for solving $x$ 
from given data~$y$ is ill-posed. 
In the following, $\NA$ and $\RA$ denote the nullspace and 
the range of an operator~$A$, respectively. By $A^\dagger$ we denote the pseudoinverse 
of $A;$ cf., e.g., \cite{EHN96}. We symbolize norm-convergence by 
$\to$ and weak convergence by $\rightharpoonup$.
We denote the weak limit by the symbol $\wlim,$ and, for a closed subspace $Z$, 
$\Pi_Z$ denotes the associated orthogonal projector onto $Z$. 
In the following, we  assume (unless specified otherwise) the attainable case
 for problem \eqref{main}, i.e., that $y$ is in $\RA$. 
In this case, we can set $y$ being the  image of an element in $N(A)^\bot.$
\[ A \xd = y, \qquad \xd \in \NA^\bot. \]
It is the unique element $\xd,$  which we want to reconstruct from given data $y$. 

We are interested in projection methods acting as a regularization, i.e., 
in approximating the pseudoinverse of $A$ by solving discrete least-squares 
problems related to  \eqref{main}. For this task we introduce  discretizations in 
the spaces $X$ and $Y$. Precisely, we assume given an increasing sequence of 
 finite-dimensional spaces $X_n \subset X$ and $Y_m \subset Y,$ $n,m \in \N$, with the property
\begin{equation}\label{projdef} 
X_{n} \subset X_{n+1}, \qquad \overline{\bigcup_n X_{n}} = X,
\qquad 
Y_{m} \subset Y_{m+1}, \qquad \overline{\bigcup_m Y_{m}} = Y.
\end{equation}
For the discretization spaces we always denote the associated orthogonal projector onto $X_n$ by $P_n := \Pi_{X_n}$ and 
onto $Y_m$ by~$Q_m := \Pi_{Y_m}$ 
\[ P_n :X \to X_n, \qquad Q_m: Y \to Y_m.  \] 

The discretization of \eqref{main} by a general projection method involves the operator 
\begin{equation}\label{genporo} \Amn:= Q_m A P_n, \end{equation}
and we define the associated solutions by (assuming attainability)
\begin{equation}\label{defxmn} \xmn := \Amnd y = \Amnd A \xd. \
\end{equation}
It is well-known that $\xmn$ is the unique solution of minimum norm under all 
least-squares solutions of the projected problem, i.e., 
\begin{equation*}
\begin{split}
\xmn  &= \mbox{argmin}_{x \in X_n} \|Q_m A x - y\|^2 = 
\mbox{argmin}_{x \in X_n} \|Q_m A x - Q y\|^2,  \quad 
 \mbox{ and } \\
 \xmn &\in N(\Amn)^\bot.
 \end{split}
\end{equation*}
It follows that  $\Amnd = P_n \Amnd Q_m$. The  general projection method \eqref{defxmn}
embraces two special well-known methods: if we put formally $m=\infty$, and hence $Q_m = I$, 
we obtain the projected least-squares method involving 
\begin{equation}\label{mnls} \xn := \And A \xd, \qquad \An = A P_n. \end{equation}
Conversely if we set $n = \infty$ and formally put $n = \infty,$ we obtain the 
dual least-squares method, 
\[ \xm := \Amd A \xd, \qquad \Am = Q_m A . \] 
We distinguish these important special cases by labeling them with only 
one index for the first method and by the index $\mim$ for the second one. 
However, the dual least-squares method is not so much of interest for this paper
(although it is of practical importance) as it always leads to a convergent method.

It is clear that $\Amnd$ is a bounded operator and hence $\xmn$ can be computed 
in a stable way. Moreover, the usual rules for adjoints and inverses hold:
$(\Amnd)^* = \Amnsd.$

The immediate question that arises from this setup is, if $\xmn$ in \eqref{defxmn} 
converges to $\xd$ as 
$n,m \to \infty$, in what sense does this convergence happen, and for which $\xd$ does this hold. 

More precisely, we study two different subjects: 
\begin{itemize} 
\item {\bf Local convergence.}
Fix $\xd$. Find conditions such that 
\begin{align*} 
 &\xmn \rightharpoonup \xd  \qquad \text{ as } n,m \to \infty, 
 \\[-4mm]
 \intertext{or}  \nonumber \\[-6mm]
 &\xmn \to  \xd   \qquad \text{ as } n,m \to \infty\,.  
\end{align*}

\item {\bf Global convergence.} Find conditions such that 
 \begin{align*} 
   &\xmn \to  \xd   \qquad \text{ as } n,m \to \infty  \qquad \forall \xd  \in N(A)^\bot . 
 \end{align*}
 
\end{itemize}

The second issue concerns convergence not only for one fixed $\xd$ 
but for all $\xd \in N(A)^\bot$. 
 
Both questions are relevant for the general projection method \eqref{defxmn} and the 
projected least-squares method \eqref{mnls}.
Note that the distinction between weak and strong convergence is not relevant for global convergence because the corresponding 
conditions are identical \cite{GrNe88}.

Of course, these question have been discussed and partly answered in literature, but  often only for the
projected least-squares method or even with further restriction like injective operators; see Section~\ref{sec:rev}
for a review. It is observed that many authors in different articles use different conditions to prove convergence 
of a specific scheme, for instance,  \eqref{ubc0} or \eqref{natterer} below. The relation between different  conditions 
in different papers  is not always obvious. It is one of the purposes of this paper to clarify this 
situation and  to unify the convergence conditions at best, 
to generalize known results to the general projection case using the operator \eqref{genporo} and hereby avoiding
unnecessary assumptions like injectivity. 

Let us mention that the convergence of  $\xmn$ to $\xd$ is the most important requirement for 
the projection methods discussed here to act as regularization. The second one, the stability of the regularization,
is automatically satisfied since we are dealing with finite-dimensional problems. Indeed, if  convergence 
of $\xmn$ to $\xd$ is verified, it is not difficult to find error estimates for noisy data as well and 
with appropriate parameter choice rules (where the index of the approximation spaces $n,m$ act as ``regularization parameter''),
convergence of $\xmn$ to $\xd$ can be proven even for the case of noisy data. We do not dwell further on this matter
since it can be treated by  standard methods; for results on the noisy case  or 
also nonlinear problems, see, e.g., \cite{N1,N2,N3,N4,N5,N6,Ha1,Ha2,Ha3}; for combination with regularization, see e.g., 
\cite{T1,T2,T3,TiAr}. For results with focus on the analysis of specific advanced method of choosing the discretization 
spaces (like adaptivity or multilevel-type), we refer to \cite{Ad1,Ad2,Ad3}.

This paper is organized as follows: in Section~\ref{rec:two}, we review existing convergence results and prove 
some important lemmas. In Section~\ref{sec:three} we provide new conditions for global or local convergence 
and relate them to  results in literature. In Section~\ref{sec:four} we state a nontrivial example of a non-convergent 
sequence  $\xmn$ which is bounded. We summarize with a conclusion in Section~\ref{sec:five}.

\section{Known and preliminary results}\label{rec:two} 
In this section we  give an 
extensive literature review of know results related to  the questions raised in the previous
section. Moreover, we present some lemmas needed later for the convergence analysis.

\subsection{A review of known results}\label{sec:rev}
The question of local or global convergence has, of course, been 
addressed in several articles. However, as stated above, 
quite often only injective operators, i.e., $N(A) = \emptyset$,
or the case of projected least-squares problems, i.e., $Q_m = I,$ have been addressed. 
Moreover, although those results are useful, they are not always completely sharp.

Before we come to the positive results, we remind of a well-known negative 
result of non-convergence. 
The following statement is the famous counterexample of Seidman \cite{se80} for the projected least-squares problem.
\begin{example}[Seidman]
There exists a compact injective linear operator $A$ and a $\xd$ 
such that  $\xn$ as given by \eqref{mnls} is a unbounded sequence. 
Thus,  in particular, we have non-convergence  $\xn \not \to \xd$. 
Moreover there also exists $A,\xd$ as before such that 
$\xn$ is bounded but $\xn \not\to \xd.$
\end{example}
The operator used for this example is a diagonal operator in the $l^2$-sequence space
with a rank-1 perturbation:
\[ A: l^2 \to l^2 \qquad A = diag(\gamma) + \beta \otimes e_1, \]
where $\gamma$ and $\beta$ are appropriate sequences and $e_1$ is the sequence 
with all $0$ except at the first position, where it is $1$. 
By an appropriate (constructive) choice of $\xd$ and $\beta,\gamma$, 
the unboundedness of $\xn$ can be shown; see  \cite{se80} or \cite{EHN96}. 
The last statement in this theorem of a bounded (strongly-) non-convergence sequence is
stated in \cite{se80} but not explicitly proven.

Concerning the question of finding conditions for 
global convergence, the problem is well-studied. 
The following result is proven by Nashed \cite{Na76}, (for $\An$),  see also \cite{Na77},
in \cite[Theorem 3.7]{Ki11} for $\An$ being injective, and 
for the general case with $\Amn$ by Du \cite{Du08} (see also \cite{DuDu14}). 
It gives a necessary and sufficient condition for global convergence. 
\begin{theorem}\label{kith}
\[ \xmn \to \xd \quad \text{ as } m,n\to \infty  \qquad \forall \xd \in N(A)^\bot, \]
if and only if there exists a constant $C$ such that 
\begin{equation}\label{ubc0} \sup_{n,m} \|\Amnd A\| \leq C. 
\end{equation}
\end{theorem}\vspace{-2mm}
Below, we will also reprove the corresponding result (Theorem~\ref{th15}) 
and, in particular, study  characterizations of the uniform boundedness condition \eqref{ubc0};
see Theorems~\ref{lemmafive} and \ref{lemmasix}. 

Note that in \cite[Theorem 2.6]{Du08}, Theorem~\ref{kith} has been generalized to the case of 
nonatainable data, i.e., when $y = A \xd + R(A)^\dagger$. In this case the 
necessary and sufficient conditions for $\xn$ being strongly (weakly) convergent to $A^\dagger y$ is 
\eqref{ubc0} and $\Ans Q_n y \to 0 \Rightarrow \And Q_n y \to   (\rightharpoonup) 0.$ 

If follows immediately from Theorem~\ref{kith}  for the dual projection case, i.e., $\Amn = \Am,$ 
by $\Amd = \Amd Q_m$, that condition \eqref{ubc0} is always satisfied,
i.e., this method always globally converges. This is well-known and has been shown, e.g., 
in~\cite{EHN96}.

A widely used sufficient condition for uniform boundedness
and hence global convergence of $\xn$ 
has
been presented by Natterer \cite{Na77} using a result by Nit\-sche~\cite{Ni70}.
\begin{theorem}[Natterer]
Let $A$ be injective. 
Suppose that  there exists a constant~$C$ such that 
for all $\xd \in X$  there exists a $u_n \in X_n$: 
\begin{equation}\label{natterer}
\|\xd-u_n\| + \|A_n^\dagger\| \|A (\xd-u_N)\| \leq C \|\xd\|.
\end{equation} 
Then $\xn \to \xd$ as $n\to \infty$ for all $\xd \in X.$
\end{theorem}
In this theorem, 
\[ \|A_n^\dagger\|  = 
\sup_{\|\An \xn\|= 1, \xn \in X_n} \|\xn\|  = \sigma_{\rm min}^{-1}(A P_n).\] 
It is not difficult to verify that 
 \eqref{natterer} implies \eqref{ubc0},  We will generalize this result by giving a condition resembling \eqref{natterer}
which is equivalent to \eqref{ubc0} and hence yields global convergence in the general 
case (including non-injective operators and for the general projection case); see below Proposition~\ref{prnat}.

Furthermore a quite general condition has been proposed by Vainikko and H\"amarik \cite{VaHa85}
(see also \cite{Ha1,Ha2,Ha3} and \cite{HaAvGa02} and the references therein).
\begin{theorem}[Vainikko and H\"amarik]
Suppose that $N(Q_mAP_nA^*) = \{0\}.$ 
If there is a constants $C$ such that  
\begin{equation}\label{VaHa}
 \|A^* Q_m A P_n z\| \leq C \|P_n A^* Q_m A P_n z \| \qquad \forall z \in X_n, 
\end{equation}
then $\xmn \to \xd$ for all $\xd \in N(A)^\bot$.
\end{theorem}
Further results, e.g., on appropriate parameter choice  rules, are proven in  \cite{VaHa85} as well. 
We will show below (Theorem~\ref{lemmasix}) that \eqref{VaHa} is actually equivalent to \eqref{ubc0}.

A simple condition involving the product of the ill-posedness and approximation 
rate has been used by several authors (e.g., \cite{KaO12,MaSc08}) 
\begin{theorem}\label{thsimplie}
If  $\|A (I-P_n)\| \|\Amnd\| < \infty, $
then $ \xmn \to \xd$  as ${m,n \to \infty}$ for all $\xd \in N(A)^\bot$.
If for a specific $\xd,$ 
$\lim_{m,n\to \infty} \|A (I-P_n)\xd \| \|\Amnd\| = 0,$
then  $ \xmn \to \xd$  as ${m,n \to \infty}.$
\end{theorem}

The previous results are ones that hold uniformly for all $\xd$ 
(except for the very last one), and 
convergence for all $\xd \in N(A)^\bot$ is obtained. However, it is  of high 
interest to study conditions for convergence for one specific $\xd,$ when we do not 
care about global convergence. There are some statements concerning 
local convergence in literature. 

In a quite general situation, necessary and sufficient conditions for local convergence
have been established by Groetsch and Neubauer \cite{GrNe88,EHN96}. 

\begin{theorem}[Groetsch and Neubauer, also Du]\label{GN}
We have the following local convergence conditions for strong convergence:
 \begin{align*} 
  \xn \to  \xd & \Longleftrightarrow   \limsup_n \|\xn\| \leq \|\xd\|.
 \end{align*}

Moreover, suppose that 
\begin{equation}\label{spacecond}
\overline{\bigcup_n (N(A) \cap X_n)} = N(A). 
\end{equation}
Then we have the following local convergence conditions for weak convergence:
\begin{align} 
 \xn \rightharpoonup \xd &\Longleftrightarrow \sup_n \|\xn\| < \infty.  \label{gnweak}
 \end{align}
\end{theorem}
Note that the part on weak convergence in Theorem~\ref{GN}, identity \eqref{gnweak} is erroneously stated in \cite{GrNe88,EHN96} without 
the space condition \eqref{spacecond}, as has been noted by Du \cite{Du08}. 
The characterization of strong convergence is 
valid without \eqref{spacecond} and was already stated in \cite{GrNe88} (using the incomplete result for weak convergence). 
It has been rigorously proved by Du and Du \cite[Remark 4.3]{DuDu14}.
We will extend Theorem~\ref{GN} to  general projection methods with $\xmn$; see Theorem~\ref{th25}.

Besides the convergence result of Neubauer and Groetsch, a sufficient condition for strong local convergence 
without requiring information about $\xd$
has been stated by Luecke and Hickey
\cite{LuHe85}.
\begin{theorem}[Luecke, Hickey] \label{luheth}
Suppose that 
\begin{equation}\label{luhe}
\sup_n \|(\An^\dagger)^* \xn \| < \infty, 
\end{equation}
then $\xn \to \xd$. 
\end{theorem}

This result is also proven in \cite{EHN96}, where it is also explained that 
\eqref{luhe} is quite strong (and thus not a necessary condition for convergence) as it leads to 
a convergence rate of $\xn -\xd$. In Proposition~\ref{above90} we provide a similar result 
but by employing  weaker conditions.

A subtle and important point is the space condition \eqref{spacecond}.
In case of injective operators, of course, \eqref{spacecond}
holds true but in the general case not always, not even if $N(A)$ is finite-dimensional. 
(Think, for instance, of a discretization space $X_n$ that is disjoint to $N(A)$.)
We note that $(N(A) \cap X_n)$ is an increasing family of closed subspaces, 
thus the following identity holds, (cf., e.g., \cite[Chpt. 1, § 12]{Ha51}) 
\[ \overline{\bigcup_n (N(A) \cap X_n)}^\bot = \bigcap_n  (N(A) \cap X_n)^\bot,
\]
so that \eqref{spacecond} is equivalent to 
\begin{equation}\label{spacecondprime}
{\bigcap_n (N(A) \cap X_n)}^\bot = N(A)^\bot. 
\end{equation}

A recent preprint \cite[Theorem 1.1]{DuDu14} discusses equivalent conditions to 
\eqref{spacecond} (respectively, \eqref{spacecondprime}). 
\begin{theorem}\label{thfunf}
The condition \eqref{spacecond} is equivalent to each of the following conditions, 
Here ${\cal G}(A)$ denotes the graph of an operator $A$. 
\begin{itemize}
\item \begin{equation*}
\forall x \in N(A):  \lim_{n\to \infty}  \inf_{z_n \in N(\An)} \|x - z_n\| \to 0,  \end{equation*}
\item $$ \forall (x,y) \in {\cal G}(A^\dagger):  \lim_{n\to \infty}  \inf_{(z_n,w_n) \in  {\cal G}(\And)} 
\|(x,y) - (z_n,y_n)\| \to 0,  $$ 
\item for all sequences $y_n$: 
\[  \sup_n \|\And y_n\| < \infty \mbox{ and } 
y_n \rightharpoonup y:   \Rightarrow y \in {\cal D}(A^\dagger) \mbox{ and } A^\dagger y = \And y. \]
\end{itemize}

Moreover \eqref{ubc0} is equivalent to the following two conditions holding simultaneously,  
\eqref{spacecond}  and  ${\rm gap}(R(A^*A P_n), R(A^\dagger A P_n) <1,$
where the gap between two spaces $M,N$ is defined as (see \cite[Lemma 3.2]{DuDu14}
${\rm gap}(M,N) =\|P_M - P_N\|$.
\end{theorem}
We will extend the second part of this result and we show that  \eqref{ubc0} can be equivalently be characterized
as a certain angle (or gap) between subspaces (but not those in this theorem)
to be smaller then one; cf. Lemma~\ref{lemang}. Moreover, we also study local convergence also when \eqref{spacecond} is not satisfied.

The subtle fact that boundedness of $\xn$ is {\em not enough} for weak
convergence and that an additional condition, e.g., like \eqref{spacecond}, 
is needed,  is not very well-known. Du \cite[Example 2.10]{Du08} gave a counterexample 
of a sequence of $\|\xn\|$ being uniformly bounded but which does not converge weakly to $\xd$. 

\begin{example}[Du] \label{alb}
There exists a linear operator $A$ and a $\xd \in N(A)^\bot$ such that 
\[ \sup_n \|\xn\| < \infty, \]
but $\xn \not \rightharpoonup \xd$. 
In this example, however, $\xn$ converges strongly to some element $\not = \xd$.
\end{example} 
The operator in this counterexample is actually not ill-posed but a simple 
projection operator onto the complement of a one-dimensional subspace, $A = I - (.,e) e$
with some appropriately chosen $e$. Failure of convergence happens because
 $\xn$ converges (even strongly) but to the ``wrong'' solution.

Below in Theorem~\ref{neucount}, we give a counterexample that is even more extreme: 
a situation like in the previous result, Example~\ref{alb}, but where 
the sequence $\xn$ does \emph{not} converge at all (not even weakly). 
This example has been devised by Neubauer \cite{Nepr}.
\begin{example}[Neubauer]
There exists a linear operator $A$ and a $\xd \in N(A)^\bot$ such that 
\[ \sup_n \|\xn\| < \infty, \]
but $\xn$ does not converge weakly. 

Moreover, the sequence $\xn$ has a subsequence, which converges weakly but 
with limit $u \not = \xd$, and no weakly convergent subsequence has limit $\xd$.
\end{example}

\subsection{Preliminary lemmas}
Since $\xmn$ is always in $N(\Amn)^\bot$, it is important to study these spaces.
We note that by discretization,  $\Amn$ is an operator with closed range. 
We have the well-known duality relations,
\begin{equation}\label{dual} R(\Amn^*) = N(\Amn)^\bot  \quad \text{ and } \quad 
 N(\Amn^*) = R(\Amn)^\bot, \end{equation}
and all these spaces are closed.

The following characterizations follow easily. 
\begin{lemma}\label{lemmaone}
\begin{align}
N(\Amn) &= 
 \{ x \in X \,| x  = w_n + q_n : w_n \in X_n \cap N(Q_mA), q_n \in X_n^\bot \}   
 \label{onen} \\ 
\begin{split}  
N(\Amn)^\bot &= \{ x \in X_n \cap  (N(Q_mA)\cap X_n)^\bot \} \\
&  = \{ x \in X_n | \exists  v_n: x = P_n A^* Q_m v_n \},  \nonumber
\end{split}
\end{align}
\end{lemma}
\begin{proof}
The first identity \eqref{onen} follows easily 
by $x = P_n x + (I-P_n)x$. 
It is straightforward to proof that any $x \in X_n \cap (X_n \cap N(Q_m A))^\bot$  
is in $N(\Amn)^\bot$. Conversely, by taking either $w_n$ or $q_n$ to be $0,$
it follows that any $x \in N(\Amn)^\bot$ must be in both $X_n$ and $(X_n \cap N(Q_m A))^\bot$.
The last identity is \eqref{dual}.
\end{proof}
Note that these spaces are not necessarily nested. 
However, the following inclusions can be  verified:
\begin{align} N(Q_m A) \cap X_n & \ \subset \   N(Q_m A) \cap X_{n+1} \ \subset  \ldots 
\ N(Q_m A),  \label{nested1} \\
N(Q_m A) \cap X_n &\ \supset \ N(Q_{m+1} A) \cap X_{n} \ \supset \ldots 
\ N(A) \cap X_n, \label{nested2} \\
 N(A) \cap X_n &= \bigcap_m \left(N(Q_m A) \cap X_n\right). \label{spxx}
\end{align}
A simple consequence is that, for a fixed $x$,  the norm of the projection\linebreak
$\|\Pi_{N(Q_m A) \cap X_n} x\|,$ is increasing in $m$ (for fixed $n$) and 
decreasing in $n$ (for fixed $m$).

We now state some approximation results, namely that 
elements in $N(A)^\bot$, i.e., the space where $\xd$ lives, can be 
approximated arbitrary well by elements in the corresponding discrete 
space $N(\Amn)^\bot$. 
 
\begin{lemma}\label{ellem}
For all $x \in N(A)^\bot$, 
\begin{equation}\label{app}
\lim_{n,m\to \infty} \inf_{z\in R(\Amn^*)} \|x - z\| =
\lim_{n,m\to \infty} \inf_{z\in N(\Amn)^\bot} \|x - z\| = 0\, ,
\end{equation}
\end{lemma}
\begin{proof}
Let $x \in N(A)^\bot = \overline{R(A^*)}=\overline{R(A^*A)}$. 
Then for any $\epsilon>0$ fixed, we 
can find a $w_{\epsilon}$ such that 
\[ \|x - A^*w_\epsilon \| \leq \epsilon, 
.\]
Moreover by \eqref{projdef}, 
with $\epsilon$ and  $w_{\epsilon}$ as before,
we can find a $n_0$  such that for all $n \geq n_0$ and $m \geq n_0$
\[ \|P_n A^*w_{\epsilon} - A^*w_{\epsilon} \| \leq \epsilon, \quad 
\|Q_m w_{\epsilon} -w_{\epsilon}\| \leq \frac{\epsilon}{\|A^*\|}, \]
thus, 
\[ \|x - P_n A^*Q_m w_{\epsilon}\| \leq 
\|x - P_n A^* w_{\epsilon}\|  + \| P_n A^* (I-Q_m) w_\epsilon\| \leq 
2 \epsilon. \]
which yields \eqref{app}. 
Note that  $P_n A^*Q_m w_{\epsilon}$ is in $R(\An^*) = N(A_n)^\bot.$
\end{proof}

We remark that $N(\Amn)^\bot$ is not necessarily a subspace of $N(A)^\bot$ so that 
it  is not correct to say that $N(\Amn)^\bot$ is dense in $N(A)^\bot$.

Furthermore, the operators $\Upmn$ and $\Upmns$ will play an important role in the subsequent analysis. 
\begin{lemma}
\begin{equation*}
N(\Amn) = N(\Upmn) \qquad 
R(\Upmns)
= R(\Amns)\end{equation*}
\end{lemma} 
\begin{proof}
Clearly $N(\Amn) \subset N(A^*\Amn) = N(\Upmn).$ Conversely 
for $x \in N(\Upmn)$ we have $A^* Q_m A P_n x = 0$, and hence also 
$P_n A^*Q_m  Q_m A P_n x =0$, thus \mbox{$x \in N(\Amn),$} which shows $N(\Amn) = N(\Upmn),$
and  by \eqref{dual} the lemma follows.  
\end{proof}
Using $\Upmn$ and $\Upmn$, we have a characterization of the solution operator 
$\Amnd A$ as a certain nonorthogonal projection operator.
\begin{proposition}\label{propang}
For any $x \in X$ and any $n,m \in \N,$  we have the unique decomposition 
\begin{equation}\label{undec1} x = v_{n,m} + u_{n,m} \qquad v_{n,m} \in R(\Upmns), u_{n,m} \in N(\Upmns). 
\end{equation}
Moreover the mapping $x \to v_{n,m}$ is given by $\Amnd A$, i.e., 
\[ v_{n,m} = \Amnd A x.  \]
For any $x \in X$ any $n,m \in \N,$  we have the unique decomposition
\begin{align} 
x &= \overline{v}_{n,m} + \overline{u}_{n,m} \qquad \
\overline{v}_{n,m} \in R(\Upmn), \overline{u}_{n,m} \in N(\Upmn) \label{undec2a} \\
& = \overline{v}_{n,m} + \overline{w}_{n,m} + \overline{q}_{n,m} \quad 
  \begin{array}{l} \overline{v}_{n,m} \in R(\Upmn),\\[1mm]  \overline{w}_{n,m} \in N(Q_mA) \cap X_n,\\[1mm]
  \overline{q}_{n,m} \in X_n^\bot. \end{array} \label{undec2b}
\end{align}
Here, $\overline{v}_{n,m},\overline{u}_{n,m}, \overline{w}_{n,m}, \overline{q}_{n,m}$ 
are uniquely determined. 
Moreover, the mapping $x \to \overline{v}_{n,m}$ is given by $A^*\Amnsd,$ i.e., 
\[ \overline{v}_n = A^*\Amnsd x, \]
and the mapping $x \to  \overline{w}_{n,m}$ is given by 
\begin{equation}\label{subsub} \overline{w}_{n,m} = \Pi_{N(Q_mA)\cap X_n} x. 
\end{equation}
\end{proposition}
\begin{proof}
Define $v_{n,m} = \Amnd A x$, then $v_{n,m} \in N(\Amn)^\bot = 
N(\Upmn)^\bot = R(\Upmns).$ 
In particular, we have  $v_{n,m} \in X_n.$
Moreover $v_n$ satisfies the normal equations 
\[0 =  P_n A^* Q_m (A P_n v_n - A x) =  
 P_n A^* Q_m A (v_n - x). \] 
Thus $v_n - x \in N(\Upmns)$ yielding the desired decomposition \eqref{undec1}. 
Conversely, for any other decomposition as above, it follows that 
$v_n \in N(\Amn)^\bot,$ and it satisfies the normal equations. 
By  uniqueness of the minimal-norm least-squares solution, 
if follows that $v_n = \Amn^\dagger A x$. Thus the decomposition is unique.
For the second part, define 
$z_{n,m} = \Amnsd x,$ then 
$z_{n,m} \in N(\Amns)^\bot = R(\Amn)$. 
In particular $z_{n,m} \in Y_m.$ 
The normal equation implies that $Q_m A P_n A^*  z_{n,m} - Q_mA P_n x =0$, thus 
$\overline{v}_{n,m} - x = A^*z_{n,m} - x \in N(\Amn) = N(\Upmn)$, which gives the 
decomposition. 
Any other decomposition of the form 
$x = A^* \Amn p + N(\Amn)$ implies that $\Amn p$ satisfies the same
normal equation as $z_{n,m}$ and it clearly  is in $R(\Amn) = N(\Amns)^\dagger$, thus 
by the uniqueness of the minimal-norm least-squares solution, we have
$\Amn p = \Ands x.$ Hence, $A^* \Amn p = A^*\Ands x,$ which implies the unique decomposition \eqref{undec2a}.  
The decomposition of 
$\overline{u}_{n,m},$ into  $\overline{w}_{n,m} +\overline{q}_{n,m}$  exists by 
the characterization of $N(\Amn)$ in \eqref{onen} and is clearly unique since 
$\overline{w}_{n,m}$ is orthogonal to $\overline{q}_{n,m}.$
Since $ \overline{v}_{n,m} = A^* z_{n,m} = A^* Q_m z_{n,m},$ it follows that 
$\overline{v}_{n,m}$ is orthogonal to  $N(Q_m A)$ and clearly 
$\overline{q}_{n,m}$ is orthogonal to $X_n,$ hence 
 applying  $\Pi_{N(A)\cap X_n}$ to the 
decomposition gives the representation for $\overline{w}_{n,m}.$
\end{proof}

\begin{remark}
This proposition will be used widely; in particular, we recognize that 
$\xmn = \Amnd A \xd$ is the first element in \eqref{undec1} in the 
decomposition of $\xd$ and thus  $\xmn$ is  the result of 
a nonorthogonal (oblique) projection applied to $\xd$.

From \eqref{undec1}, \eqref{undec2a} we also obtain the nontrivial fact that 
\[ N(\Upmns) \cap R(\Upmns) = \emptyset, \quad \text{ and} \quad 
N(\Upmn) \cap R(\Upmn) = \emptyset.
\]
\end{remark} 
As a corollary we have a formula for $\Pi_{N(Q_m A)\cap X_n}.$ 

\begin{corollary}\label{cor10}
For any $k \leq n$, 
\[ P_k \Pi_{N(Q_m A)\cap X_n} = P_k - A_k^* \Amnsd,  \]
\[ \Pi_{N(Q_m A)\cap X_n}P_k  = P_k - \Amnd A_k. \]
\end{corollary}
\begin{proof}
Applying $P_k$ to \eqref{undec2b}
and using the fact that orthogonal projectors are selfadjoint and 
$(A_k^* \Amnsd)^* = \Amnd A_k$ yields the result. 
\end{proof}

As another illustration of the usefulness of Proposition~\ref{propang}, we can 
prove a similar  characterization of the space condition \eqref{spacecond}
as in Theorem~\ref{thfunf}.

\begin{proposition}
We have that \eqref{spacecond} is satisfied if and only if 
\begin{equation}\label{xxx} \left(\lim_{n,m \to \infty}\inf_{z_{n,m} \in R(\Amns)}  \|x - z_{n,m}\| = 0 \right) \Longrightarrow x \in \overline{R(A^*)} 
\end{equation}
\end{proposition}
\begin{proof} 
Let \eqref{xxx} hold and suppose that \eqref{spacecond} does not hold. Then there exists a 
$x \not = 0$ and 
$x \in N(A)$ and $x \in \bigcap_n \left( N(A) \cap X_n\right)^\bot.$
For such a $x$ using \eqref{undec2b}, \eqref{subsub}, and \eqref{spxx}, it follows that 
for all $n$ $\lim_{m\to \infty} \overline{w}_{n,m} = 0$. Thus, 
by Corollary~\ref{cor10} with $k = n$, 
$$\lim_{m\to \infty} \|(P_n - \Ans \Amnsd)x\| = 
\lim_{m\to \infty} \|(P_n - \Amnd \Amn)x \| = 0.$$ 
Taking $n$ such that $\|x - P_n x\| \leq \epsilon,$ we find a $m$ such that 
$$\|x - \Amnd \Amn x \| \leq  \|x - P_n x\| + \|(P_n - \Amnd \Amn)x \| \leq 2 \epsilon. $$ 
Thus by \eqref{xxx}, since  $\Amnd \Amn \in R(\Amns),$ it follows that $x \in \overline{R(A^*)} = N(A)^\bot$.
Since $x \in N(A),$ we have a contradiction, thus \eqref{spacecond} must hold. 
Conversely, if \eqref{spacecond} holds, suppose that  \eqref{xxx} does not hold. 
Then we have a $x \in N(A)$ and $z_{n,m} \in R(\Amns)$ with $\|x - \Pi_{R(\Amns)} x\| \to_{n,m} 0.$
As  $ \Pi_{R(\Amns)} =\Amnd \Amn = \Amns \Amnsd,$ we have that $x -  \Amns \Amnsd \to_{n,m} 0$.
Applying $P_n$ to  \eqref{undec2b}, it follows that 
$\overline{w}_{n,m}  = P_n x - \Amns \Amnsd x \to_{n,m} 0$; in particular 
$\lim_{n} \lim_{m} \overline{w}_{n,m} = 
\lim_{n} \Pi_{N(A) \cap X_n} x = 0.$  By \eqref{nested1}, $\Pi_{N(A) \cap X_n} x$ is 
increasing, hence $\Pi_{N(A) \cap X_n} x = 0$ for all $n$. In other words, 
$x \in \bigcap_n (N(A) \cap X_n)^\bot.$ Using \eqref{spacecond} implies that $x \in N(A)^\bot,$
which is a contradiction to $x \in N(A)$. 
\end{proof}

In view of Lemma~\ref{ellem}, we always have  that 
\begin{equation}\label{malmol} N(A)^\bot \subset \left\{ x: \text{dist}(x, N(\Amn)^\bot) \to_{n,m} 0 \right\},
\end{equation}
but according to \eqref{xxx}, equality  holds only if the space condition \eqref{spacecond} holds. 
The corresponding result for the projected least-squares case using $\An$ (and more) 
has already  been proven in \cite{DuDu14}.

\section{Convergence results}\label{sec:three}
We now study necessary and sufficient conditions for local and global
convergence of $\xmn$ to $\xd$, thus extending  the known results of 
the Section~\ref{sec:rev}.

\subsection{Conditions for global convergence}
At first we consider convergence for all $\xd  \in N(A)^\bot.$
We reprove the statement of Theorem~\ref{kith} 
based on the following lemma.
\begin{lemma}\label{next}
For all $\xd \in N(A)^\bot$ and $\xmn =  \And A \xd,$ we have 
\begin{align*}
\limsup_{n,m \to \infty} \|\xmn -\xd\| &\leq  
\limsup_{n,m \to \infty} \|\Amnd  A (I-P_n) \xd\| 
\end{align*}
\end{lemma}
\begin{proof}Noting 
that $\Amnd \Amn = I - \Pi_{N(\Amn)} = \Pi_{N(\Amn)^\bot} = \Pi_{R(\Amns)},$
\begin{align*} &\xmn -\xd = \Amnd A \xd - \xd = 
 (\Amnd Q_m A P_n - I) \xd + \Amnd Q_m A (I-P_n) \xd \\
 & \qquad = 
 -(I - \Pi_{R({\Amns})})\xd + \Amnd  A(I-P_n) \xd. 
\end{align*}
Thus we have
\begin{align*} \|\xmn -\xd\| &\leq  
 \inf_{z \in R({\Amns})}\|\xd - z\|  + 
\|\Amnd  A (I-P_n) \xd\| 
\end{align*}
Now, Lemma~\ref{ellem} and \eqref{app} yields the result 
\end{proof}

\begin{remark} In the above result  we can easily replace 
$\Amnd  A (I-P_n) \xd$ by the expression
\mbox{$\Amnd Q_m A \Pi_{N(Q_m A)^\bot} (I-P_n) \xd$} or 
by $\Amnd Q_m A \Pi_{N(Q_m A)^\bot \cap X_n^\bot} \xd$. 
\end{remark}

We obtain  the first (well-known) result on global convergence; cf. Theorem~\ref{kith}. 
\begin{theorem}\label{th15}
The approximations $\xmn$ converge
to $\xd$ for all $\xd \in N(A)^\bot$ if and only if 
there exists a constant $C$ such that 
\begin{equation}\label{ubc}  \sup_{n,m} \|\Amnd A\| \leq C.  \end{equation}
Equivalent to \eqref{ubc} is that there exists a constant $C'$ such that
\begin{equation}\label{ubc2}  \sup_{n,m} \|\Amnd A (I-P_n)\| \leq C.  \end{equation}
\end{theorem}
\begin{proof}
Let $\xmn \to \xd$ for all $\xd \in N(A)^\bot$ as $m,n \to \infty$. Then we have
by $\xmn = \Amnd A \xd$ that
$\Amnd A \to  I$ pointwise on $N(A)^\bot.$
By the uniform boundedness principle this implies that
$\Amnd A|_{N(A)^\bot}$ must be uniformly bounded. 
But it is easy to see that this is equivalent to \eqref{ubc}. 

Conversely let \eqref{ubc} hold, then 
\[  \|\Amnd A (I-P_n) \xd\| \leq 
 \|\Amnd A\|\| (I-P_n) \xd\|, \]
and by \eqref{ubc} and \eqref{projdef}, the first result follows.
Since $ \And A (I-P_n) =  \Amnd A - \Amnd \Amn,$  and $\Amnd \Amn$ is 
always bounded,  \eqref{ubc2} follows. 
\end{proof}

Next, we  study condition~\eqref{ubc} in depth and rewrite it in other forms. 
Clearly it holds that
\begin{equation}\label{aaaa} \sup_{m,n} \|\Amnd A\| \leq C  \Leftrightarrow
 \sup_{m,n} \|A^*\Amnsd\| \leq C. 
\end{equation}
We show that \eqref{ubc} is equivalent to the fact that 
 angles between certain subspaces are uniformly bounded. 
More precisely, we consider the norm of the product of orthogonal projectors
onto two subspaces, which is related to the (minimal canonical) angle; 
cf.  \cite[Lemma~5.1]{Sz06}.

\begin{lemma}\label{lemang}
For a sequence of  closed subspaces $X_n,Y_n,$ 
let $\Pi_{X_n} ,\Pi_{Y_n} $ be the corresponding 
orthogonal projectors.

Then we have the following equivalent conditions 
\begin{align} &\exists \rho <1:\,   \sup_n \|\Pi_{X_n} \Pi_{Y_n}\| \leq \rho,  \label{firsta} \\ 
\Leftrightarrow &
\exists \tau > 0\, \forall x \in X_n, y \in Y_n: \, \|(x+y)\|^2 \geq \tau \|x\|^2 , 
\label{secaa}
\\
\Leftrightarrow  &
\exists \tau '> 0\, \forall x \in X_n, y \in Y_n:\,  \|(x+y)\|^2 \geq \tau \|y\|^2  .
\label{thirda}
\end{align}
\end{lemma}
\begin{proof} 
If \eqref{firsta} holds, then 
by Young's inequality for any $\epsilon >0$ $x \in X_n,$ $y \in Y_n,$
\[ \|(x+y)\|^2 \geq  \|x\|^2 + \|y\|^2 - 2 \rho \|x\|\|y\|
 \geq \|x\|^2(1 - \rho \epsilon) + \|y\|^2 (1 -\tfrac{\rho}{\epsilon} ).
\]
With $\epsilon = \rho$ or $\epsilon = \rho^{-1}$ either 
of \eqref{thirda} or \eqref{secaa}  follows.
Conversely let \eqref{secaa} hold, then 
 then 
for any $x \in X_n$, $y \in y_n$ with $\|x\| = \|y\| = 1$ 
and any $\epsilon>0$ 
\[
\epsilon^2 \tau^2 \leq \|\epsilon x + \frac{1}{\epsilon} y\|^2  = 
 \epsilon^2  + \frac{1}{\epsilon^2}
- 2 (x,y).
\]  
Taking $\epsilon^2 = (1-\tau^2)^{-\frac{1}{2}}$ gives the bound 
\[ 2(x,y)  \leq 2 (1-\rho^2)^\frac{1}{2} < 2, \]
thus with 
(cf. \cite[Lemma~5.1]{Sz06}) 
\[  \left\|\Pi_X \Pi_Y \right\| = \sup_{\|x\|\leq 1 , \|y\| \leq 1, x \in X, y \in Y}
(x,y), \]
the result follows. 
\end{proof}
 
Combining Proposition~\ref{propang} and Lemma~\ref{lemang} yields the following 

\begin{theorem}\label{lemmafive}
The uniform boundedness condition \eqref{ubc} is 
equivalent to one of the following (and hence all) conditions:
\begin{align}
\exists \eta <1: \forall n: \qquad \|\Pi_{N(\Upmns)} \Pi_{R(\Upmns)}\| &< \eta,&    \label{cond} \\
\exists \eta <1: \forall n: \qquad \|\Pi_{N(\Upmn)} \Pi_{R(\Upmn)} \|  &< \eta,& \label{condadjp} \\
\exists \eta <1: \forall n: \qquad \|(I-P_n)  \Pi_{R(\Upmn)} \| &< \eta.& \label{condadj}  
\end{align}
\end{theorem} 
\begin{proof}
The boundedness condition can be rephrased as the condition that 
a constant $C$ exists with (using the notation in Proposition~\ref{propang})
\[ \|v_{n,m}\| \leq C \|x\| = C \| v_{n,m} + u_{n,m} \|. \] 
for all $v_{n,m} \in R(\Upmns)  $ and $u_{n,m} \in N(\Upmns)$. 
However, this is \eqref{secaa} with 
the spaces $N(\Upmns)$  and $R(\Upmns)$, thus Lemma~\ref{lemang} gives  \eqref{cond}. 
By \eqref{aaaa} we have the equivalent characterization of uniform boundedness 
using \eqref{secaa} that a constant exists, such that 
\[ \|\overline{v}_{n,m}\| \leq C \|x\| = C \| \overline{v}_{n,m} + \overline{u}_{n,m} \|, \] 
which yields \eqref{condadjp}. 
By \eqref{thirda}, this is equivalent to 
the existence of a constant such that 
\begin{equation}\label{help1} \|  \overline{w}_{n,m} + \overline{q}_{n,m} \| \leq C \|x\|, 
\end{equation}
with $\overline{w}_{n,m}$, $\overline{q}_{n,m}$ as in  Proposition~\ref{propang}. 
However, $\overline{w}_{n,m}$  is always uniformly bounded for bounded $x$
by \eqref{subsub}, and it is orthogonal to $q_{n,m}$. 
Thus,  this condition is satisfied if and only if 
$$\|\overline{q}_{n,m} \|^2 \leq C' \|x\|^2 = 
C'(\|\overline{v}_{n,m}+ \overline{q}_{n,m} + \overline{w}_{n,m}\|^2) = 
C'(\|\overline{v}_{n,m}+ \overline{q}_{n,m}\|^2 +\|\overline{w}_{n,m}\|^2).$$ 
Since we can take $x =  \overline{v}_{n,m}+ \overline{q}_{n,m} + \overline{w}_{n,m}$ with 
arbitrary chosen elements $\overline{v}_{n,m},\overline{q}_{n,m},\overline{w}_{n,m}$ out of the
corresponding spaces, we have that \eqref{help1} holds 
if and only if for all $\overline{v}_{n,m} \in R(\Upmn)$ 
and all $\overline{q}_{n,m} \in X_{n}^\bot$ 
\[ \|q_{n,m} \| \leq  C'(\|\overline{v}_{n,m}+ \overline{q}_{n,m}\|), \]
which is equivalent to  \eqref{condadj}. 
\end{proof} 

\begin{remark}
We remark that for closed subspaces $X$, $Y$, in Hilbert spaces 
 the identity 
$\|\Pi_X \Pi_Y\| <1 \Leftrightarrow \|\Pi_{X^\bot} \Pi_{Y^\bot}\| <1$  
usually does \em{not} hold \cite{CoMa10,De95}.
\end{remark}

These conditions can be rewritten in more convenient form. 
\begin{theorem}\label{lemmasix}
The uniform boundedness condition \eqref{ubc} is 
is equivalent to one (and hence all) of the following conditions: 
\begin{align}
  &\exists C >0 : \forall n,m,x   &  \|(I-P_n)A^*Q_m A P_n x\| \leq  C'  \|P_n A^*Q_m A P_n x \|, 
  \label{threeA} \\
&\exists \eta <1: \forall n,m,x  &
 \|(I-P_n) A^*Q_m A P_n x\| \leq  \eta \|A^*Q_m A P_n x \|,  
\label{oneA} \\
&\exists \eta <1: \forall n,m,w  &  
  \inf_{v} \| P_n A^*Q_m A w - A^*Q_m A P_n v\|
\leq \eta \| P_n A^*Q_m A w \|,  \label{fourA}  \\
 &\exists \eta <1: \forall n,m,v &   \inf_{w} \| P_n A^*Q_mA w - A^*Q_mA P_n v\|
\leq \eta \| A^*Q_m A P_n v \|. \label{fiveA} 
\end{align}
\end{theorem}
\begin{proof}
Condition \eqref{condadj} can be rewritten as  \eqref{oneA}. 
By splitting the terms 
using the complementary orthogonal projectors $P$ and $I-P$, 
it is easy to see that this is equivalent to 
\eqref{threeA}.
The identities \eqref{fourA} and \eqref{fiveA} are 
\eqref{cond} and \eqref{condadjp}, respectively, 
when writing the projectors onto the nullspaces as 
complementary projectors onto the ranges of the adjoints and 
using the minimization property of such orthogonal projectors. 
\end{proof}

It is not difficult to verify that \eqref{threeA} is equivalent to Vainikko and H\"amarik's condition~\eqref{VaHa}.
Note that a characterization over angles of subspaces has also been used 
by Du and Du \cite[Theorem 1.2]{DuDu14} for the case $Q_m = I$ and  with different spaces, which do not yield 
an equivalent condition  to \eqref{ubc0} but need additionally the space condition \eqref{spacecond}.

\subsubsection{Necessary and sufficient conditions for convergence}\label{necsuff}
In this section we investigate practically useful conditions such that 
\eqref{ubc} is satisfied and necessary conditions for \eqref{ubc}.

We  find a condition of Natterer's type that is equivalent to the uniform 
boundedness condition extending Natterer's result to the cases of 
 non-injective operators and  $Q_m\not = I$.
\begin{proposition}\label{prnat}
The uniform boundedness condition \eqref{ubc} and 
hence $\xmn \to \xd$ for all $\xd \in N(A)^\bot$ holds if and only 
if there exists a constant $C$ such that for all $\xd \in X$ there exists 
a $u_n \in X_n$ such that 
\begin{equation}\label{nc}\|\xd - u_n\| + \|\Amnd A (\xd - u_n)\| \leq C \|\xd\|. 
\end{equation}
In particular for the case $Q_m = I$, if Natterer's condition \eqref{natterer} holds for all 
$\xd \in N(A)^\bot$, then 
 \eqref{ubc} holds, thus $\xn \to \xd$ for all $\xd \in N(A)^\bot$.
\end{proposition}
\begin{proof}
If \eqref{ubc} and hence global convergence holds, then $u_n = \xmn$ satisfies 
\eqref{nc}. Conversely if \eqref{nc} holds, then 
since $Q_m A u_n = \Amn u_n$
\begin{align*} \|\Amn^\dagger A x \| & \leq 
 \|\Amnd Q_m A (x -u_n) \| +  \|\Amnd \Amn u_n\| \\
& \leq   \|\Amnd A (x -u_n) \| + 
c_1 \| u_n\| \\ & \leq 
\|\Amn^\dagger A (x -u_n) \| + 
c_1 \| u_n-\xd \| + c_1 \|\xd\| \leq C' \|\xd\| \, .
\end{align*}
It is easy to see that \eqref{natterer} implies \eqref{nc}.
\end{proof} 

From the conditions \eqref{threeA}--\eqref{fiveA}, probably 
\eqref{threeA} is the most useful. 
We introduce the norm of the pseudoinverse of the discretized forward operator 
\[ \|\Amnd\| = \frac{1}{\sigma_{\rm min}(Q_m A P_n)} = 
  \sup_{x \in P_n, Q_m A P_n x \not = 0}  \frac{(x,x)}{(x,P_n A^* Q_m A P_n x)},
\]
where $\sigma_{\rm min}$ denotes the smallest (by definition nonzero) singular value. 
We have the following result:
\begin{lemma}
If there exists a constant $C$  such that
\begin{equation}\label{thisaa} 
\forall n,m \qquad  \|(I-P_n) A^*Q_m \| \|\Amnd\| \leq C, 
\end{equation}
then \eqref{threeA} and hence \eqref{ubc} is satisfied. 
\end{lemma}
\begin{proof}
In view of \eqref{threeA}, we observe that 
\[ \|P_n A^*Q_m z\| \geq \sigma_{\rm min}(Q_m A P_n) \|z\|. \]
Taking $z = A P_n x$ and $\|(I-P_n)A^* Q_m AP_nx\| \leq 
\|(I-P_n)A^*Q_m \| \|Q_m A P_n x\|$ proves the assertion.
\end{proof}
Note that this result implies  in particular Theorem~\ref{thsimplie}.
In the same way we could prove the result by replacing \eqref{thisaa} by 
\begin{equation*}
\forall n,m \qquad  \|(I-P_n) A^*Q_mA\| \|\Amnd\|^2 \leq C. 
\end{equation*}

Natterer \cite{Na77} has outlined how to prove conditions like \eqref{nc} 
in practical situations, namely from inverse inequalities of approximation 
spaces combined with error estimates for the approximation. 
Using \eqref{thisaa} we can do a similar thing. 
\begin{proposition}
Let $(H_s,\|x\|_s)_{s \in \R}$ be a Hilbert scale generated by a densely defined unbounded 
selfadjoint strictly positive operator $L$, i.e.,
\[ \|x\|_s = \|L^s x\|. \]
Suppose that $A:H_0 \to H_0$ is such that 
for some numbers $l,r>0$ 
\begin{equation}\label{zero} 
c_1 \|x\|_{-l}  \leq  \|A x\|_{L^2} \leq c_2 \|x\|_{-r} \qquad \forall x \in H_0, 
\end{equation} 
and that $X_n$ is a discrete subspace satisfying 
the approximation condition 
\[ \|(I -P_n)z\| \leq \gamma_n \|z\|_{l} \qquad \forall z \in H_l , \]
and the inverse inequality 
\[ \|z_n\|_{L^2} \leq \frac{1}{\beta_n} \|z_n\|_{-r} \qquad \forall z_n \in X_n \]  
holds. Then if  
\[\limsup_{n}  \frac{\gamma_n}{\beta_n}   \leq C, \]
the uniform boundedness condition \eqref{ubc} holds. 
\end{proposition}
\begin{proof}
We have that 
\[ \|A P_n x\| \geq c_1 \|P_n x\|_{-l} \geq  c_1 \beta_n\|P_n x\|. \]
Thus, 
\[ \|A_n^\dagger\| \leq \frac{1}{c_1 \beta_n}. \]
Moreover with $z = A^* x$ we find that 
\[ \|(I-P_n)A^*x \| \leq \gamma_n \|A^* x\|_{r}. \]
From the right hand side of \eqref{zero}, 
we see that $A L^{r}$ is a bounded linear operator and 
so is its adjoint 
$ L^r A^*$, i.e., $ \|A^* x\|_r \leq C \|x\|.$ 
Thus \eqref{thisaa} is satisfied by 
\[ \|(I-P_n)A^*x \| \|A_n^\dagger\| \leq 
C \frac{\gamma_n}{\beta_n}  \leq C. \] 
\end{proof}

In a typical case of finite-element spaces or spline spaces and 
if we consider a Hilbert scale of Sobolev spaces, then 
the inverse inequality is usually satisfied with $\beta_n = \frac{1}{n^l}$ and
the approximation condition with $\gamma_n = \frac{1}{n^r}$. 
Thus if $r = l$ is applicable, then we obtain convergence. 
A similar argument has been utilized by Natterer using condition
\eqref{natterer}.

The next result concerns the dual variant of \eqref{condadj}.
\begin{proposition}
If 
\begin{equation}\label{wiederwas}
\exists: \eta <1 :\forall n: \qquad  \|P_n \Pi_{N(\Upmns)} \| < \eta,  
\end{equation}
then \eqref{ubc} is satisfied. 
Moreover, \eqref{wiederwas} holds if 
\begin{equation}\label{wwa} \exists: \eta <1 :\forall n, \forall w: \qquad 
\inf_v \|A^*Q_mA P_n v - P_n w\| \leq \eta \|P_nw \|. 
\end{equation}
\end{proposition}
\begin{proof}
Inequality \eqref{wiederwas} can be written as 
\[ \eta >  \|\Pi_{N(\Upmns)} P_n \| = 
\|(I-\Pi_{R(\Upmn)}) P_n \|.  \]
By the characterization of orthogonal projectors as minimizers we have that
this is equivalent to 
\[ \inf_{v \in R(\Upmn)} \| v - P_n x\| \leq \eta \|P_n x\|, \]
which is exactly \eqref{wwa}.
Setting $x = A^*Q_mA w$ we obtain that this implies \eqref{fourA}.
\end{proof}

Note that \eqref{wiederwas} is not equivalent to \eqref{ubc} because 
\eqref{wiederwas} can only hold if the intersection of the corresponding
spaces is empty. However, if \mbox{$X_n \cap N(\Upmns) \not = \emptyset$}, then 
\eqref{wiederwas} cannot hold but \eqref{ubc} still can.

Let us now come to a necessary condition for uniform boundedness. 
We show that 
the uniform boundedness \eqref{ubc} implies the space condition 
\eqref{spacecond}. In the case $Q_m = I,$ this 
has already been observed by Du \cite{Du08}. 
\begin{proposition}
Let \eqref{ubc} hold, then 
\[ \bigcap_n (N(A) \cap X_n)^\bot = N(A)^\bot, \]
i.e., the space condition \eqref{spacecond} holds. 
\end{proposition}
\begin{proof}
Since $\bigcup_n (N(A) \cap X_n) \subset N(A)$ it follows that 
$N(A)^\bot \subset \bigcap_n (N(A) \cap X_n)^\bot$.  Thus, we only
need to proof the opposite inclusion. Let $x \in \bigcap_n (N(A) \cap X_n)^\bot.$
In view of \eqref{spxx} we have that for all $n$, 
\[ \lim_{m \to \infty } \|x - \Pi_{\left(N(Q_m A) \cap X_n\right)^\bot} x \| \to 0, \] 
thus using  \eqref{undec2b} for $x,$  we have that 
\[ \forall n: \qquad \lim_{m\to \infty}  \overline{w}_{n,m} = 0 . \]
By \eqref{nested2}, we  have that the double sequence 
$\|\overline{w}_{n,m}\|$ is  decreasing in $m$ for all $n$. 
Since 
\[ \|x\|^2 = \|\overline{w}_{n,m}\|^2 + \|\overline{v}_{n,m} + \overline{q}_{n,m}\|^2, \]
we have that for all $n,$ $\|\overline{v}_{n,m} + \overline{q}_{n,m}\|$ is 
increasing in $m$ and that\linebreak \mbox{$\lim_{m \to \infty} \overline{v}_{n,m} + \overline{q}_{n,m} = \xd.$}
Thus for all $n,$ $\sup_m \|\overline{v}_{n,m} + \overline{q}_{n,m}\| = \xd,$
and hence $\|\overline{v}_{n,m} + \overline{q}_{n,m}\|$ is bounded uniformly in $n,m$. 
Since \eqref{ubc} implies \eqref{condadj} using \eqref{thirda}, we have 
a constant $C$ such that 
\[ \|\overline{q}_{n,m}\| \leq C \| \overline{v}_{n,m} + \overline{q}_{n,m}\|  \leq C \|x\|.  \]
Thus,  $\|\overline{q}_{n,m}\|$ is uniformly bounded, it has a weakly convergent subsequence
as $n,m \to \infty,$  and as $q_{n,m} \in X_n^\bot$, it follows that this limit can only be $0$.
By a subsequence argument 
we conclude that $\wlim_{n,m \to \infty} \overline{q}_{n,m} = 0.$ 
It follows that the iterated limit $\wlim_{n \to \infty} \left( \wlim_{m \to \infty} \overline{q}_{n,m} \right) = 0.$ 
Thus, 
\[ x = \wlim_{n \to \infty} \left( \wlim_{m \to \infty}  \overline{v}_{n,m} \right). \] 
Since each $v_{n,m}$ is in $N(A)^\bot$, and this space is weakly closed, all the limits are 
in $N(A)^\bot$ as well, thus $x \in N(A)^\bot$. 
\end{proof}

\subsection{Conditions for local convergence}\label{secloc}
We are now interested in local convergence results, i.e., 
to study the question if for a given a
specific element $\xd \in N(A)^\bot$ the corresponding sequence 
$\xmn$ converges (weakly or strongly). The difference to the previous 
section is that 
the conditions imposed here 
are not ``uniform'' in $\xd$ but depend on the specific $\xd$.

A practically useful sufficient condition for strong convergence is a 
simple consequence of Lemma~\ref{next} (compare Theorem~\ref{thsimplie}).
\begin{proposition}
If 
\[ \lim_{m,n\to \infty} \|\Amnd\| \|A (I- P_n) \xd\| \to 0, \]
then $\xmn \to \xd $ as $m,n\to \infty$.
\end{proposition} 
Hence if $\xd$ can be approximated well in $X_n$, we can hope for strong 
convergence. This result is quite crude compared to Theorem~\ref{GN}, where
(except for  weak convergence)  equivalent conditions to convergence are established. 
However, the mentioned  theorem of  Groetsch and Neubauer (with the supplementary result of Du)
can also be extended with minor modifications to the case of $Q_m \not = I$. 
\begin{theorem}\label{th25}
We have the following local convergence conditions for strong convergence.
 \begin{align} \label{strongnm}
  \xmn \to  \xd & \Longleftrightarrow   \limsup_{n,m} \|\xmn\| \leq \|\xd\|.
 \end{align}

Suppose that the space condition \eqref{spacecond} holds.
Then we have the following local convergence conditions for weak convergence.
\begin{align*} 
\xmn \rightharpoonup \xd &\Longleftrightarrow \sup_{n,m} \|\xmn\| < \infty. 
 \end{align*}
\end{theorem}
\begin{proof}
Consider first the part on weak convergence. By boundedness, $\xmn$ has a weakly
convergent subsequence with limit $u.$ As in \cite{GrNe88} it follows immediately 
that $u -\xd \in N(A).$ Moreover each $\xmn \in (N(Q_m A) \cap X_n)^\bot \subset (N(A) \cap X_n)^\bot,$
thus  $\xmn  \in (N(A) \cap X_k)^\bot$ for all $k \geq n.$ This implies that 
$u \in \bigcap_n (N(A) \cap X_n)^\bot$, and by \eqref{spacecond}, $u \in N(A)^\bot$. 
Thus $u-\xd \in N(A) \cap N(A)^\bot,$  hence $u = \xd$. By a subsequence argument, 
$\xmn \rightharpoonup \xd.$ 
For \eqref{strongnm}, we do not need \eqref{spacecond}. The proof follows \cite{Du08}:
from \eqref{strongnm}, we again find a weakly convergence subsequence with limit $u$ and  
 $u -\xd \in N(A)$ and $\xd \in N(A)^\bot$. Thus, 
\[ \|u-\xd\|^2 + \|\xd\|^2 = \|u\|^2 \leq 
\liminf_{n,m} \|\xmn\|^2 \leq \limsup_{n,m} \|\xmn\|^2 \leq \|\xd\|^2, \]
thus $u = \xd$ and as before $\xmn \rightharpoonup \xd.$ From \eqref{strongnm} we also find that 
$\|\xmn\| \to \|\xd\|$, which together with weak convergence implies strong convergence. 
The other directions of the implications are trivial. 
\end{proof}

In a next step, we  replace \eqref{spacecond} by other ``local'' conditions. 
\begin{lemma}
We have that 
\begin{equation}\label{null}
\xmn \rightharpoonup \xd \qquad \Longleftrightarrow 
\begin{cases} \sup_{n,m} \|\xmn\| < \infty & \text{ and}  \\
 \Pi_{N(A)} \xmn \rightharpoonup 0.  \end{cases} 
 \end{equation} 
\end{lemma} 
\begin{proof}
Suppose that $\xmn$ converges weakly to $\xd$. Then for arbitrary $z$ 
\begin{align*}  \lim_{n,m \to \infty} (\Pi_{N(A)} \xmn ,z) &= 
 \lim_{n,m \to \infty} ( \xmn , \Pi_{N(A)}z) \\
 &=  \lim_{n,m \to \infty} ( \xmn -\xd, \Pi_{N(A)}z) = 0,
 \end{align*}
where we used that  $(\xd, \Pi_{N(A)}z) = 0$ as $\xd \in N(A)^\bot$.
Conversely, let $\xmn$ be bounded and  $\Pi_{N(A)} \xmn \rightharpoonup 0.$ 
As in the proof before, $\xmn$ has a weakly convergent subsequence with limit $u$
such that $u-\xd \in N(A).$
Thus with $z = u-\xd$ we have that 
\begin{align*}  0 &= \lim_k ( \Pi_{N(A)} \xmnk ,u-\xd) = 
 \lim_k (\xmnk  , \Pi_{N(A)}(u-\xd))\\
 &= 
  \lim_k (\xmnk  -\xd , \Pi_{N(A)}(u-\xd)) = 
  \|u-\xd\|^2, 
  \end{align*}
thus $u = \xd$.  By a subsequence argument $\xmn \rightharpoonup \xd.$
\end{proof}

As a consequence, we can find sufficient conditions for weak and strong local convergence
generalizing  the results  of Luecke and Hickey.
\begin{proposition}\label{above90}
Suppose that with  some fixed constant $C,$ there exists   for each $n,m$ 
an index pair $(\tilde{m}_{n,m},\tilde{n}_{n,m}) \in \N\times \N$  with 
$\lim_{n,m\to \infty} \tilde{n}_{n,m} = \infty,$  and $\tilde{m}_{n,m} \geq m,$
and $\tilde{n}_{n,m} \leq n,$  
such that 
\begin{equation}\label{bcc}
\sup_{n,m} \|\xmn\| < C \quad \text{and} \quad 
\sup_{n,m}  \|A^* \Amndsind{\tilde{m}_{n,m},\tilde{n}_{n,m}} \xmn\| < C\, .
\end{equation}
Then $\xmn \rightharpoonup \xd$ as $m,n\to \infty$.

If we can choose $(\tilde{m}_{n,m},\tilde{n}_{n,m}) = (n,m)$, i.e., 
\begin{equation}\label{bbc1}
\sup_{n,m} \|\xmn\| < C \quad \text{and} \quad 
\sup_{n,m}  \|A^* \Amndsind{n,m} \xmn\| < C,
\end{equation}
then $\xmn \to \xd$ as $m,n\to \infty$.
\end{proposition}
\begin{proof}
We apply \eqref{undec2b} and get for all $n,m$
\[ \xmn =A^* \Amndsind{\tilde{m}_{n,m},\tilde{n}_{n,m}} \xmn + \overline{w}_{\tilde{m}_{n,m},\tilde{n}_{n,m}} + 
\overline{q}_{\tilde{m}_{n,m},\tilde{n}_{n,m}}. \]
By \eqref{bcc} it follows that $\overline{w}_{\tilde{m}_{n,m},\tilde{n}_{n,m}} + 
\overline{q}_{\tilde{m}_{n,m},\tilde{n}_{n,m}} $ is uniformly bounded, and  
since these two elements are orthogonal to each other  it follows that both components are uniformly bounded as well, hence
they have weakly convergent subsequences  as $n,m \to \infty$ with limit $w,q$.

For fixed $k$,  $P_k q =
\wlim_{n,m \to \infty} P_k \overline{q}_{\tilde{m}_{n,m},\tilde{n}_{n,m}} = 0$ since $\tilde{n}_{n,m} \to \infty,$ and 
 thus it follows that $q =0$. 
Since $\xmn \in (N(Q_{\tilde{m}_{n,m}} A) \cap X_{n})^\bot,$ for $\tilde{m}_{n,m} \geq m,$ 
and $(N(Q_{\tilde{m}_{n,m}} A) \cap X_{n})^\bot \subset (N(Q_{\tilde{m}_{n,m}} A) \cap X_{\tilde{n}_{n,m}})^\bot,$
we have that $\overline{w}_{\tilde{m}_{n,m},\tilde{n}_{n,m}} = 0$. 
Thus we have for a subsequence 
\[ \wlim_{k \to \infty} \left(\xmnk -A^* \Amndsind{\tilde{m}_{m_k,n_k},\tilde{n}_{m_k,n_k}} \xmnk  \right) = 0. \]
By a subsequence argument we have that this holds for the whole sequence. 
\begin{equation}\label{blalab}  
 \wlim_{n,m \to \infty} 
(\xmn - A^* \Amndsind{\tilde{m}_{n,m},\tilde{n}_{n,m}} \xmn) = 0.   
\end{equation}
It follows that 
\[ \Pi_{N(A)} \xn =  \Pi_{N(A)} (\xmn - A^* \Amndsind{\tilde{m}_{n,m},\tilde{n}_{n,m}} \xmn)  \rightharpoonup 0,
\qquad \text{as } m,n\to \infty,
\]
thus, by \eqref{null} we obtain the result that $\xn \rightharpoonup \xd$ as  $m,n\to \infty.$ 
Since \eqref{bbc1} is a special case of \eqref{bcc}, we have that under 
\eqref{bbc1}, $\xmn$ converges weakly to $\xd$ 
and furthermore by \eqref{blalab}
also  that $ A^* \Amndsind{n,m} \xmn \rightharpoonup \xd$  as  $m,n\to \infty.$ 
Thus, by weak convergence, 
\[ \lim_{n,m\to \infty} \|\xmn\|^2 = \lim_{n,m \to \infty} \left( A^* \Amndsind{n,m} \xmn,\xd \right) = (\xd,\xd)  = \|\xd\|^2.\]
By the Radon-Riesz property, we obtain that $\|\xmn -\xd\| \to 0.$ 
\end{proof}

\begin{remark}
Proposition~\ref{above90} includes  Theorem~\ref{luheth} as a special case.
Indeed, setting $Q_m = I$, from \eqref{luhe}, the boundedness of $\xn$ and 
$\|A^* \Andsind{n} \xn\|$ follows immediately, and thus by \eqref{bbc1} 
we obtain Theorem~\ref{luheth} as a corollary.
\end{remark}

\section{A counterexample}\label{sec:four}
In this section we provide a nontrivial example of a sequence of 
projected least-squares solutions, 
$\xn,$ which is  bounded but non even weakly convergent. 
Note that Du's example considers a similar situation but the sequence 
$\xn$ is strongly convergent (but not to $\xd$). 
The example again stresses the importance of the space conditions \eqref{spacecond}
and the fact that the part on weak convergence in Theorem~\ref{GN} is false 
without the space condition \eqref{spacecond}.

\begin{theorem}\label{neucount}
There exists an operator $A$ and $\xd$ and  a sequence of finite-dimensional spaces  $(X_n)_n$ 
satisfying \eqref{projdef}
such that 
$x_{2n} \rightharpoonup u$ but $x_{2n} \not\to u$ and  
$x_{2n +1} \rightharpoonup v,$ but $x_{2n+1} \not \to v$ and $u,v \not = \xd$.
In particular the sequence $\xn$ neither converges weakly nor strongly
and it has no weakly convergent subsequence has limit $\xd$.
\end{theorem}
\begin{proof}
Let $X$ be a separable Hilbert space with orthonormal basis $e_{ij}$, $(i,j)\in\N
\times\N$, i.e., all elements $x\in X$ may be represented via
\[ x=\sum_{i,j=1}^\infty\xi_{ij}e_{ij}\qquad\text{with}\qquad \|x\|^2=
   \sum_{i,j=1}^\infty\xi_{ij}^2<\infty\,. \]
We define a linear bounded operator $A:X\to X$ via
\begin{equation*}
 Ax:=\sum_{i=1}^\infty\sum_{j=2}^\infty(\xi_{ij}+q^j\xi_{i1})e_{ij}\,,
\end{equation*}
where $q\in(0,1)$ is fixed. Obviously,
\begin{equation*}
 \|Ax\|^2=\sum_{i=1}^\infty\sum_{j=2}^\infty(\xi_{ij}+q^j\xi_{i1})^2\le
 2\max\left\{1,\frac{q^4}{1-q^2}\right\}\|x\|^2\,.
\end{equation*}
It is easy to see that $A$ has an infinite-dimensional nullspace. It holds that
\begin{equation}\label{nsp} x\in N(A) \Longleftrightarrow  \forall  i\ge 1,j>1:\,\xi_{ij}=-q^j\xi_{i1}\, \text{ and } 
   \sum_{i=1}^\infty\xi_{i1}^2<\infty\,. 
\end{equation}   
A generalized solution $\xd=A^\dagger y$ of the equation $Ax=y$ is always an element
of $N(A)^\bot$;
we may characterize these elements as follows:
\begin{equation}\label{perp}
 z=\sum_{i,j=1}^\infty\eta_{ij}e_{ij}\in N(A)^\bot \Longleftrightarrow \forall i\ge 1:\,
 \begin{array}{l} \displaystyle \eta_{i1}=
 \sum_{j=2}^\infty q^j\eta_{ij}\,\text{ and }\\[3mm]
 \displaystyle \sum_{i=1}^\infty\sum_{j=2}^\infty
 \eta_{ij}^2<\infty\, \end{array}. 
\end{equation}
Now we choose finite-dimensional subspaces of $X$:
\[ X_n:=\text{span}\{e_{ij}:1\le i,j\le n\}. \]
Obviously, \eqref{projdef} holds. 
Let $\xd\in N(A)^\bot$ with $\xd =\sum_{i,j=1}^\infty\xdeta_{ij}e_{ij} \in N(A)^\bot$  
and $y:=A\xd,$ and set $\xn:=\And y$, where $\An:=AP_n.$ 

Since, due to \eqref{nsp}, $N(A)\cap X_n=\{0\}$, we get by \eqref{onen} 
that $N(\An)=X_n^\bot$. Therefore,
\begin{equation}\label{defxnn}  \xn=\sum_{i,j=1}^n\xi_{ij}^ne_{ij} \end{equation}
is the unique minimizer in $X_n$ of the problem
\[ \|A\xn-y\|^2=\sum_{i=1}^\infty\sum_{j=2}^\infty\left( (\xi_{ij}^n-\xdeta_{ij})+q^j
  (\xi_{i1}^n-\xdeta_{i1}) \right)^2\,\to\,\min\,, \]
where $\xi_{ij}^n:=0$ if $i>n$ or $j>n$.

From the first order necessary conditions for a minimum 
we obtain the solution:
\begin{equation*}
 \xi_{ij}^n=\xdeta_{ij}-q^j(\xi_{i1}^n-\xdeta_{i1})\,,\qquad\text{if}\;1\le i\le n\,,
 2\le j\le n\,,
\end{equation*}
and
\begin{align*}
 \xi_{i1}^n &= \xdeta_{i1}+\left(\sum_{j=n+1}^\infty q^{2j}\right)^{-1}
 \sum_{j=n+1}^\infty q^j\xdeta_{ij} \nonumber \\
 &= \xdeta_{i1}+(1-q^2)q^{-(n+1)}\sum_{j=0}^\infty q^j\xdeta_{i,n+1+j}\,,\qquad
 1\le i\le n\,. 
\end{align*}

Now we choose a concrete element $\xd \in N(A)^\bot$:
\[ \xdeta_{ij}:=\frac{q^j}{1-q^2}(c_i\rho_j+r_{ij})\,,\qquad i\ge 1,\,j>1\,, \]
with
\begin{equation}\label{condN}
 \rho_j:=\begin{cases} 1\,, & j\text{ even}\,, \\ 0\,, & j\text{ odd}\,,\end{cases}
 \qquad r_{ij}:=\begin{cases} 1\,, & j=i+1\,, \\ 0\,, & \text{ else}\,,
 \end{cases} \qquad\sum_{i=1}^\infty c_i^2<\infty\,,
\end{equation}
and the extension \eqref{perp} for $j = 1$.
The condition on the coefficients $c_i$ guarantees that
\[ \sum_{i=1}^\infty\sum_{j=2}^\infty{\xdeta_{ij}}^2<\infty\,. \]
It then holds that
\begin{equation}\label{xi3}
 \xi_{i1}^n=\xdeta_{i1}+c_ie_n+r_{i,n+1}\qquad\text{with}\qquad e_n:=
 \sum_{j=0}^\infty q^{2j}\rho_{n+1+j}\,.
\end{equation}
Noting that
\begin{equation}\label{en}
 e_n=\frac{1}{1-q^4}\cdot\begin{cases} q^2\,, & n\text{ even}\,, \\ 1\,, & n
 \text{ odd}\,,\end{cases} 
\end{equation}
\eqref{xi3} implies that 
\begin{align}
 \lim_{l\to\infty}\xi_{i1}^{2l} &= \xdeta_{i1}+\frac{c_iq^2}{1-q^4}, \label{xi4} \\
 \lim_{l\to\infty}\xi_{i1}^{2l+1} &= \xdeta_{i1}+\frac{c_i}{1-q^4}. \label{xi5}
\end{align}
Let us now define the two elements
\[ u:=\sum_{i,j=1}^\infty u_{ij}e_{ij}\qquad\text{and}\qquad
   v:=\sum_{i,j=1}^\infty u_{ij}e_{ij} \]
with
\begin{align*}
 u_{i1} &:= \xdeta_{i1}+\frac{c_iq^2}{1-q^4}&  u_{ij} &:= \xdeta_{ij}-
 q^j(u_{i1}-\xdeta_{i1})\,, \quad  j>1\,, \\
 v_{i1} &:= \xdeta_{i1}+\frac{c_i}{1-q^4} &  v_{ij} &:= \xdeta_{ij}-q^j(v_{i1}-
 \xdeta_{i1})\,, \quad   j>1\,.
\end{align*}
Obviously, due to \eqref{nsp} and \eqref{condN}, $u-\xd$ and $v-\xd \in N(A),$ and thus
$u$ and $v$ are least-squares
solutions of $Ax=y = A\xd$ with $u,v \not = \xd.$ 
Together with \eqref{condN}, \eqref{xi3}, and \eqref{en} we immediately obtain that
(remember that $x_n = \And y$ is given by \eqref{defxnn}, \eqref{xi3})
\begin{equation}\label{nocon} \|x_{2l}-P_{2l}u\|^2=\frac{q^4-q^{2n+2}}{1-q^2}=
   \|x_{2l+1}-P_{2l+1}v\|^2. \end{equation}
Now \eqref{xi4} and \eqref{xi5} imply that
\begin{align*}
 x_{2l} &\rightharpoonup u\, , \qquad   x_{2l+1} \rightharpoonup v\,,  
\end{align*}
but \eqref{nocon} implies $x_{2l} \not\to u$ and  $x_{2l+1} \not\to v.$
Thus, it is possible that $\xn$ has different weakly convergent subsequences,
but it neither converges weakly nor strongly towards $\xd$.
\end{proof}

\section{Conclusion}\label{sec:five}
We have studied global and local convergence of general projection
schemes for ill-posed problems. For global convergence, we have
established the uniform boundedness condition \eqref{ubc} as
being necessary and sufficient and have found concrete conditions in 
Theorems~\ref{lemmafive} and \ref{lemmasix} when this holds. 
Several practically useful sufficient condition were given in 
Section~\ref{necsuff}. Concerning local convergence, we have 
generalized the well-known results of Groetsch and Neubauer and Du 
giving an equivalent characterization of local convergence by norm bounds in Theorem~\ref{th25}. 
Further sufficient conditions of the type of Luecke and Hickey were given in 
Proposition~\ref{above90}. 

In the analysis, we point out two important findings: the recognition of 
the $\xmn$ as oblique projection of $\xd,$ which leads to a study of 
angles of a sequence of subspaces. The second point is the 
question if the intuitive identity ``$N(A)^\bot = \lim_{m,n} N(\Amn)^\bot$'' is valid, 
understood in the sense as \eqref{malmol}. As the inclusion ``$\subset$`` always 
holds, this  gives a way of applying the uniform boundedness principle. However it 
is important to notice that this identity does only hold unless the additional space condition 
\eqref{spacecond} holds. While for injective operators this is trivially true, 
for noninjective operators \eqref{spacecond} has to be taken into account when studying
local (weak) convergence. 

The issue that this condition is needed for weak convergence is illustrated by 
a nontrivial counterexample in Theorem~\ref{neucount} of a bounded sequence $\xn$
which does not converge at all, thus generalizing the examples of Seidman and Du.

\section*{Acknowledgments}
The author would like to thank Andreas Neubauer for useful discussions and for 
providing the counterexample in Theorem~\ref{neucount}.

\bibliographystyle{siam}
\bibliography{regpromain}

\end{document}